\documentclass{amsart}
\usepackage[utf8]{inputenc}
\usepackage[T1]{fontenc}

\usepackage{amsmath,amsthm,amssymb}
\usepackage{enumerate}
\usepackage{vaucanson-g}
\usepackage[all]{xy}

\numberwithin{equation}{section}
\usepackage{cleveref}
\usepackage{comment}

\def\N{\mathbb N}
\def\A{{\mathcal A}}
\def\B{{\mathcal B}}
\def\C{{\mathcal C}}
\def\P{{\mathcal P}}

\def\LR{{\mathcal{LR}}}
\def\AC{{\mathcal{AC}}}
\def\BO{{\mathcal {BO}}}
\def\u{{\mathbf u}}
\def\Ret{{\mathcal R}}

\renewcommand{\L}[1][]{\mathcal{L}_{#1}(\mathbf{u})}

\def \PE {{\mathcal{PE}}}
\def \L {\mathcal{L}(\mathbf{u})}
\def \Rk#1 {$\mathcal{R}_{#1}$}
\def \Rext {{\rm {Rext}}}
\def \Lext {{\rm {Lext}}}
\def \Bext {{\rm {Bext}}}
\def \Pext {{\rm {Pext}}}
\def \b {{\rm {b}}}

\def \FC#1 {
\mathcal{C}
\ifthenelse{\equal{#1}{}}{}{(#1)}
}

\def \PC#1 {
\mathcal{P}
\ifthenelse{\equal{#1}{}}{}{(#1)}
}

\newtheorem{thm}{Theorem}

\newtheorem{coro}[thm]{Corollary}
\newtheorem{corollary}[thm]{Corollary}
\newtheorem{lem}[thm]{Lemma}

\newtheorem{proposition}[thm]{Proposition}
\newtheorem{defi}[thm]{Definition}

\theoremstyle{definition}
\newtheorem{example}[]{Example}

\theoremstyle{remark}

\crefname{thm}{theorem}{Theorems}
\crefname{thrm}{theorem}{Theorems}
\crefname{coro}{corollary}{Corollaries}
\crefname{example}{example}{Examples}
\crefname{lem}{lemma}{Lemmas}
\crefname{lmm}{lemma}{Lemmas}
\crefname{claim}{claim}{Claims}
\crefname{obs}{obsertvation}{Observations}
\crefname{proposition}{proposition}{Propositions}
\crefname{prop}{proposition}{Propositions}
\crefname{defi}{definition}{Definitions}

\def\pf{\begin{proof}}
\def\pfk{\end{proof}}

\begin{document}
\title{Sturmian Jungle (or Garden?) on Multiliteral Alphabets}
\author{L\!'ubom\'ira Balkov\'a}
\author{Edita Pelantov\'a}
\author{\v St\v ep\'an Starosta}
\address{Department of Mathematics, FNSPE, Czech Technical
University in Prague, Trojanova 13, 120~00 Praha~2, Czech Republic}

\date{\today}
\email{lubomira.balkova@fjfi.cvut.cz, edita.pelantova@fjfi.cvut.cz, staroste@fjfi.cvut.cz}

\maketitle
\begin{abstract}
The properties characterizing Sturmian words are considered for
words on multiliteral alphabets.  We
summarize various generalizations of Sturmian words to
multiliteral alphabets and enlarge the list of  known
relationships among these generalizations. We also collect  many
examples of infinite words to illustrate differences in  the
generalized definitions of Sturmian words.
\end{abstract}

\section{Introduction}
Sturmian words, i.e., aperiodic words with the lowest factor complexity, appeared first in the paper of Hedlund and Morse in 1940.
Since then Sturmian words have been in the center of interest of many mathematicians and the number of discoveries of new properties and connections keeps growing.
The charm of Sturmian words consists in their natural appearance while studying diverse problems.
Many equivalent definitions have been found that way.
Sturmian words are binary and
every property characterizing Sturmian words asks for a fruitful extension to an analogy on a larger alphabet.
Well-known examples of such efforts are Arnoux-Rauzy words, words coding interval exchange transformations, or billiard words.
All these words belong to well established classes and their descriptions and properties can be found in many works~\cite{ArRa, Ke, Ra, FeZa, ArMaShTa, Barysh, Ta}.
An overview of some generalizations of Sturmian words is provided in~\cite{Be} and~\cite{Vu2}.


The aim of this paper is to attract attention to other generalizations of Sturmian words.
Our motivation stems from recent results on palindromes in infinite words
that have ended in the definition of words rich in palindromes, the definition of defect,
the description of a relation between factor and palindromic complexity, etc.~\cite{ABCD, BrHaNiRe, BaMaPe}.
Impulses for such an intensive research of palindromes come concededly from the article~\cite{DrPi}
which characterizes Sturmian words by palindromes, the article~\cite{DrJuPi}
which investigates the number of palindromes in prefixes of infinite words
and last, but not least, the discovery of the role of palindromes in description of the spectrum of
Schr\"{o}dinger operators with aperiodic potentials~\cite{HoKnSi}.
While generalizing Sturmian words we have taken into consideration the characterization of Sturmian words by return words from~\cite{Vu}
and a recent definition of Abelian complexity~\cite{RiSaZa, RiSaZa2},
which is closely connected with balance properties.

We consider the following properties ($k$ denotes the cardinality of alphabet $\A$):
\begin{enumerate}
\item{Property $\mathcal C$}:

the factor complexity of $\u$ satisfies
${\mathcal C}(n)=(k-1)n+1$ for all $n \in \mathbb N$.
\item{Property $\LR$}:

$\u$ contains one left special and one right special factor of every length.
\item {Property $\BO$}:

all bispecial factors of $\u$ are ordinary.

\item{Property $\Ret$}:

any factor of $\u$ has exactly $k$ return words.
\item{Property $\P$}:

the palindromic complexity of $\u$ satisfies
${\P}(n)+{\P}(n+1)=k+1$ for all $n \in \mathbb N$.
\item{Property $\PE$}:

every palindrome has a~unique palindromic extension in $\u$.
\item{Balance properties}:
\begin{enumerate}
\item{Property $\B_{\forall}$}:

$\u$ is aperiodic and for all $a \in {\A}$ and for all factors $w,v \in \L$ with $|w|=|v|$ it holds
$$||w|_a-|v|_a|\leq k-1.$$
\item{Property $\B_{\exists}$}:

$\u$ is aperiodic and there exists $a \in {\A}$ such that for all factors $w,v \in \L$ with $|w|=|v|$ it holds
$$||w|_a-|v|_a|\leq k-1.$$
\item{Property $\AC$}:

$\u$ is aperiodic and the abelian complexity of $\u$ satisfies $\AC(n)=k$ for all $n \in \N, \ n \geq 1$.
\end{enumerate}
\end{enumerate}

%

All properties are equivalent on a binary alphabet and they characterize Sturmian words.
No two of them are equivalent on the set of infinite words over a multiliteral alphabet.
The non-equivalence is shown by counterexamples.
However some properties imply others, or it can be shown that a couple of properties are equivalent on a certain class of infinite words.
For instance, on the class of uniformly recurrent ternary words Properties $\Ret$ and $\BO$ are equivalent.

There exist more equivalent definitions of Sturmian words, for instance the definition based on balance properties of subfactors of factors~\cite{FaVu},
on the index of an infinite word~\cite{MaPe}, or Richomne's characteristics of Sturmian words~\cite{Ri}.
We do not pay attention to these definitions in our survey.


The paper is organized as follows.
In \cref{sec:not_def} we recall the notions
playing an important role in the definitions of Properties 1 through 7.
We recall the notion of substitution which is irrelevant for the generalizations of Sturmian words
but is used to construct most of examples of infinite words.
\Cref{sec:opulent} is focused on the study of palindromes in infinite words:
we summarize older and new results concerning palindromes, we define palindromic branches.
A new result in this section is \Cref{thm:minus1} providing a~new characterization of rich words by means of bilateral orders.
\Cref{sec:equiv} shortly summarizes essential results on Sturmian words.
\Cref{generalizedSturm} is devoted to an overview of known relations among different generalizations of Sturmian words,
mostly from articles~\cite{BaMaPe, BuLuGlZa, GlJuWiZa, BaPeSt,RiSaZa, RiSaZa2}.
New results are in \Cref{th:mr1,th:mr2,PErichThenRet,biinf}.
The last section is a brief summary of selected relations and examples illustrating the studied Properties.


\section{Notations and definitions} \label{sec:not_def}
By $\A$ we denote a~finite set of symbols, usually called
{\em letters}; the set $\A$ is therefore called an {\em
alphabet}. A~finite string $w=w_0w_1\ldots w_{n-1}$ of letters of
$\A$ is said to be a~{\em finite word}, its length is
denoted by $|w| = n$. Finite words over $\A$ together with
the operation of concatenation and the empty word $\varepsilon$ as
the neutral element form a~free monoid $\A^*$. The map
$$w=w_0w_1\ldots w_{n-1} \quad \mapsto \quad \overline{w} =
w_{n-1}w_{n-2}\ldots w_{0}$$ is a~bijection on $\mathcal{A}^*$,
the word $\overline{w}$ is called the {\em reversal} or the {\em
mirror image} of $w$. A~word $w$ which coincides with its mirror
image is a~{\em palindrome}.

Under an {\em infinite word} $\u$ over the alphabet $\A$ we understand an infinite string
$\u=u_0u_1u_2\ldots $ of letters from $\mathcal{A}$ such that every letter of $\A$ occurs in $\u$.
We call an infinite word $\u$ {\em eventually periodic} if there exist finite words $w, v$ such that
$\u=wv^{\omega}$, where $\omega$ means `repeated infinitely many times'.
If $w=\varepsilon$, then $\u$ is said to be {\em (purely) periodic}.
If $\u$ is not eventually periodic, then we call $\u$ {\em aperiodic}.

A~finite word $w$ is a~{\em factor} of a~word $v$ (finite or
infinite) if there exist words $w^{(1)}$ and $w^{(2)}$ such that
$v= w^{(1)}w w^{(2)}$. If $w^{(1)} = \varepsilon$, then $w$ is said
to be a~{\em prefix} of $v$, if $w^{(2)} = \varepsilon$, then $w$ is
a~{\em suffix} of~$v$. We say that a~prefix or a~suffix is {\em proper} if it is not equal
to the word itself.

The {\em language} $\mathcal{L}(\u)$ of an infinite word
$\u$ is the set of all its factors. The factors of
$\u$ of length $n$ form the set denoted by
$\mathcal{L}_n(\u)$. Using this notation, we may write
$\mathcal{L}(\u)=\cup_{n\in
\mathbb{N}}\mathcal{L}_n(\u)$.

We say that the language
$\mathcal{L}(\u)$ is {\em closed under reversal} if
$\mathcal{L}(\u)$ contains with every factor $w$ also its
reversal $\overline{w}$.

An infinite word $\u$ over $\A$ is called $c$-{\em balanced}
if for every $a \in \A$ and for every
pair of factors $w$, $v$ of $\u$ of the same length
$|w|=|v|$, we have $\left||w|_{a}-|v|_{a}\right|\leq c$, where $|w|_a$
means the number of letters $a$ contained in $w$.
Note that in the case of a~binary alphabet, say ${\A}=\{0,1\}$, this condition may be rewritten in a~simpler way: an
infinite word $\u$ is $c$-{\em balanced}, if for every pair of
factors $w$, $v$ of $\u$ with $|w|=|v|$, we have
$\left||w|_0-|v|_0\right|\leq c$. We call $1$-balanced words
simply {\em balanced}.

We say that two words $w,v \in \A^{*}$ are {\em abelian equivalent} if
for each letter $a \in \A$, it holds $|w|_a=|v|_a$.
It is easy to see that the abelian equivalence defines indeed an equivalence relation on $\A^{*}$.
If $\A=\{a_1,a_2, \dots, a_k\}$, then the {\em Parikh vector} associated with the word $w \in \A^{*}$ is defined as
$$\Psi(w)=(|w|_{a_1}, |w|_{a_2}, \dots, |w|_{a_k}).$$
We call {\em abelian complexity} (as defined in \cite{RiSaZa2}) of an infinite word $\u$ the function $\AC: \N \to \N$ given by
$$\AC(n)=\# \{\Psi(w) \bigm | w \in \mathcal{L}_n(\u)\}.$$

For any factor $w\in \mathcal{L}(\u)$, there exists an
index $i$ such that $w$ is a prefix of the infinite word
$u_iu_{i+1}u_{i+2} \ldots$. Such an index $i$ is called an {\em
occurrence} of $w$ in $\u$. If each factor of $\u$
has at least two occurrences in $\u$, the infinite word
$\u$ is said to be {\em recurrent}.
It can be easily shown that each factor of a recurrent word occurs infinitely many times.
It is readily seen to see that if the language of $\u$ is closed under reversal, then
$\u$ is recurrent.
The infinite word $\u$ is said to be {\em uniformly recurrent}
if for any factor $w$ of $\u$ the distances between
successive occurrences of $w$ form a~bounded sequence.

Let $j,k$, $j<k$, be two successive occurrences
of a~factor $w$ in~$\u$. Then $u_ju_{j+1}\dots u_{k-1}$ is
called a~{\em return word} of
$w$. Return words were first studied in~\cite{Du} and~\cite{HoZa}.
The set of all return words of $w$ is denoted by $R(w)$,
\begin{equation*}\label{set_of_ret_words}
R(w)=\{u_ju_{j+1}\dots u_{k-1}\mid j,k \mbox{ being successive
occurrences of } w \mbox{ in }\u\}.
\end{equation*}
If $v$ is a~return word of $w$, then the word $vw$ is called a~{\em
complete return word} of $w$.
It is obvious that an infinite recurrent word is
uniformly recurrent if and only if the set of return words of any of
its factors is finite.

The {\em (factor) complexity} of an infinite word $\u$ is the
map $\mathcal{C}: \mathbb{N} \mapsto \mathbb{N}$, defined by
$\mathcal{C}(n)=\# \mathcal{L}_n(\u)$. To
determine the increment of complexity, one has to
count the possible {extensions} of factors of length~$n$. A~{\em
left extension} of $w \in \L$ is any letter
$a\in \A$ such that $aw \in \L$.
The set of all left extensions of a~factor $w$ will be denoted by
$\Lext(w)$. We will mostly deal with recurrent infinite words $\u$.
In this case, any factor of $\u$ has at least one left extension.
A~factor $w$ is called {\em left special} (or LS for short) if $w$ has
at least two left extensions. Clearly, any prefix of a~LS factor is LS as well. It makes therefore sense to define an {\em infinite LS branch} which is an infinite word whose all prefixes are LS factors of $\u$.
Similarly, one can define a~{\em right extension}, a~{\em right
special} (or RS) factor, $\Rext(w)$, and an {\em infinite RS branch} which is a~left-sided infinite word whose all suffixes are RS factors of $\u$.

We say that a~factor $w$ of $\u$ is
a~{\em bispecial} (or BS) factor if it is both RS and LS.
The role of BS factors for the computation of
complexity can be nicely illustrated on Rauzy graphs.

Let $\u$ be an infinite word and $n\in\N$. The {\em Rauzy graph}
$\Gamma_n$ of $\u$ is a~directed graph whose set of vertices is
${\mathcal L}_n(\u)$ and set of edges is ${\mathcal
L}_{n+1}(\u)$. An edge $e\in{\mathcal L}_{n+1}(\u)$ starts
in the vertex $w$ and ends in the vertex $v$ if $w$ is a prefix
and $v$ is a suffix of $e$, see Figure~\ref{f}.
\begin{figure}[ht]
\begin{center}
\begin{picture}(220,30)
\put(50,20){\circle*{5}} \put(210,20){\circle*{5}}
\put(60,20){\vector(1,0){140}}
\put(10,8){$w=w_0w_1\cdots w_{n-1}$} \put(170,8){$v=w_1\cdots
w_{n-1}w_n$} \put(82,25){$e=w_0w_1\cdots w_{n-1}w_n$}
\end{picture}
\end{center}
\caption{Incidence relation between an edge and vertices in a
Rauzy graph.} \label{f}
\end{figure}
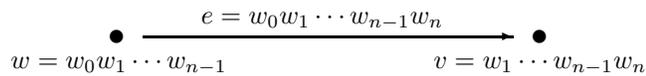
If the word $\u$ is  recurrent, the graph $\Gamma_n$ is
strongly connected for every $n\in\N$, i.e., there exists a~directed
path from every vertex $w$ to every vertex $v$ of the
graph.

If the language $\L$ of the infinite word $\u$ is closed under
reversal, then  the operation that to every vertex $w$ of the
graph associates its mirror image, the vertex $\overline{w}$, and to every edge $e$
associates $\overline{e}$ maps the Rauzy graph $\Gamma_n$ onto
itself.

The {\em outdegree} ({\em indegree}) of a~vertex $w\in{\mathcal
L}_n(\u)$ is the number of edges which start (end) in $w$.
Obviously the outdegree of $w$ is equal to $\#\Rext(w)$ and
the indegree of $w$ is $\#\Lext(w)$.
The sum of outdegrees over all vertices is equal to the number of
edges in every directed graph. Similarly, it holds for indegrees.
In particular, for the Rauzy graph $\Gamma_n$ we have
$$
\sum_{w\in{\mathcal L}_{n}(\u)}\#\Rext(w) \ =\
\mathcal{C}(n+1) \ = \sum_{w\in{\mathcal L}_{n}(\u)}\#\Lext(w)\,.
$$
The first difference of complexity $\Delta\C(n) = \mathcal{C}(n+1) - \mathcal{C}(n)$ is thus given by
\begin{equation*}\label{eq:delta}
\Delta\C(n) = \sum_{w\in{\mathcal L}_{n}(\u)}\bigl(\#\Rext(w) - 1 \bigr) \ = \sum_{w\in{\mathcal
L}_{n}(\u)}\bigl(\#\Lext(w) - 1 \bigr)\,.
\end{equation*}
A non-zero contribution to $\Delta\C(n)$ in the left-hand sum is given only
by those factors $w\in{\mathcal L}_n(\u)$ for which
$\#{\Rext}(w)\geq 2$, and for recurrent words, a non-zero contribution to
$\Delta\C(n)$ in the right-hand sum is provided only
by those factors $w\in{\mathcal L}_n(\u)$ for which
$\#{\Lext}(w)\geq 2$.  The last relation can
be thus rewritten for recurrent words $\u$ as
$$
\Delta\C(n) = \sum_{w\in{\mathcal L}_{n}(\u),\
\hbox{\scriptsize $w$ RS}\hspace*{-0.5cm}}\hspace*{-0.3cm}\bigl(\#{\Rext}(w) - 1
\bigr) \  = \sum_{w\in{\mathcal L}_{n}(\u),\
\hbox{\scriptsize $w$ LS}\hspace*{-0.5cm}}\hspace*{-0.3cm}\bigl(\#\Lext(w) - 1
\bigr)\,.
$$

If we denote ${\rm Bext(w)}=\{awb \in \L \bigm | a,b \in {\A}\}$,
then the second difference of complexity $\Delta^2\C(n) = \Delta \C(n+1)-\Delta \C(n)=\C(n+2) -2\C(n+1)+ \mathcal{C}(n)$ is given by
\begin{equation}\label{second_diff}
\Delta^2\C(n) = \sum_{w\in{\mathcal L}_{n}(\u)}\bigl(\#{\rm Bext}(w) - \#{\Rext(w)}-\#{\Lext(w)}+1
\bigr) \,.
\end{equation}

Denote by $\b(w)$ the quantity $$\b(w):= \# {\Bext}(w) - \# {\Rext}(w)-\# \Lext(w) +
1.$$ The number $\b(w)$ is called the {\em bilateral order} of the
factor $w$ and was introduced in ~\cite{Ca}. It is readily seen that if $w$ is not a~BS
factor, then $\b(w)=0$. Bispecial factors are distinguished
according to their bilateral order in the following way
\begin{itemize}
\item if $\b(w)>0$, then $w$ is a~{\em strong} BS factor,
\item if $\b(w)<0$, then $w$ is a~{\em weak} BS factor,
\item if $\b(w)=0$  then $w$ is an {\em ordinary} BS factor.
\end{itemize}

A~{\em substitution} on $\A$ is a~morphism $\varphi:{\A^{*}} \rightarrow {\A^{*}}$ such that there exists a~letter $a
\in \A$ and a~non-empty word $w \in {\A}^{*}$ satisfying
$\varphi(a)=aw$ and $\varphi(b) \not = \varepsilon$ for all $b \in \A$.
Since a~morphism satisfies $\varphi(vw)=\varphi(v)\varphi(w)$ for all $v,w
\in {\A^{*}}$, any substitution is uniquely determined by the images
of letters. Instead of classical $\varphi(a)=w$, we sometimes write
$a \to w$. A~substitution can be naturally extended to an infinite
word $\u=u_0u_1u_2\dots$ by the prescription
$\varphi(\u)=\varphi(u_0)\varphi(u_1)\varphi(u_2)\dots$ An infinite
word $\u$ is said to be a~{\em fixed point} of the substitution
$\varphi$ if it fulfills $\u=\varphi(\u)$. It is obvious that every
substitution $\varphi$ has at least one fixed point, namely $\lim_{n
\to \infty}\varphi^{n}(a)$ (to be understood in the sense of
product topology).


\section{Words opulent in palindromes} \label{sec:opulent}
In resemblance to the factor complexity ${\mathcal C}(n)$ of an infinite
word $\u$, let us define the
{\em palindromic complexity} of $\u$ as the map ${\mathcal P}:
\mathbb N \to \mathbb N$ given by
\begin{equation*}\label{PalCompl}
{\mathcal P}(n)=\# \{w \in {\mathcal L}_n(\u)| \ w=\overline{w}\}.
\end{equation*}

If $a \in {\A}$ and $w$
is a~palindrome and $awa \in \L$, then $awa$ is said to be a~{\em palindromic extension} of $w$.
The set of all palindromic extensions of $w$ is denoted by ${\rm Pext}(w)$.


Similarly as in the case of left special and right special branches, one can define a~{\em palindromic branch} of $\u$.

\begin{defi}
Let $\u$ be an infinite word.
A both-sided infinite word $v = \ldots v_3 v_2 v_1 v_1 v_2 v_3 \ldots $ is a~palindromic branch with center $\varepsilon$ of the word $\u$
if for every $n \in \N$ the word $v_nv_{n-1} \ldots v_2v_1v_1v_2 \ldots v_{n-1}v_n$ is a factor of $\u$.
Let $a$ be a letter.
A both-sided infinite word $v = \ldots v_3 v_2 v_1 a v_1 v_2 v_3 \ldots $ is a~palindromic branch with center $a$ of the word $\u$
if for every $n \in \N$ the word $v_nv_{n-1} \ldots v_2v_1av_1v_2 \ldots v_{n-1}v_n$ is a factor of $\u$.
\end{defi}

It follows from the K\"{o}nig's theorem that if $\u$ has infinitely many palindromes then $\u$ has at least one palindromic branch.
In any Sturmian word on $\left \{ 0,1 \right \}$ there exist exactly three palindromic branches with centers $\varepsilon$, $0$ and $1$.
See also Section~\ref{WellKnown}.

Uniformly recurrent words containing infinitely many distinct palindromes satisfy
that for any factor~$w$, every sufficiently large palindrome in $\u$ contains $w$, thus such a~palindrome contains $\overline w$ as well.
As a~consequence, we have the following theorem.
\begin{thm}\label{unif_rec_pal}
If $\u$ is a~uniformly recurrent word that contains infinitely many distinct
palindromes, then its language $\L$ is closed under reversal.
\end{thm}

The opposite implication is not true as illustrated by the following example.
\begin{example}[uniform recurrence + closeness under reversal $\not \Rightarrow$ infinitely many palindromes] \label{ex:finite_pals}
The infinite word $\u$ on $\{a,b\}$ (constructed in~\cite{BeBoCaFa}) whose prefixes $u_n$ are given by the following recurrent formula
$$u_0=ab, \quad u_{n+1}=u_n ab \overline{u_n},$$
is uniformly recurrent and its language is closed under reversal. However, $\u$ contains only a~finite
number of palindromes.
\end{example}

When we relax the condition of uniform recurrence, the statement of Theorem~\ref{unif_rec_pal} is not true any more.
\begin{example}[infinitely many palindromes $\not \Rightarrow$ closeness under reversal] \label{ex:pals_not_closed}

The infinite word $\u$ on $\{a,b,c\}$ whose prefixes $u_n$ are given by the following recurrent formula
$$u_0=\varepsilon, \quad u_{n+1}=u_n ab c^{n+1} u_n$$
is clearly recurrent.
Infinitely many palindromes are represented by the factors $c^n$ for every $n$.
As the factor $ba$ does not occur, the set of factors is not closed under reversal.

The word $\u$ may be recoded to a~binary alphabet while preserving the mentioned properties.
We may for instance recode $\u$ using the following mapping:
$$ a \rightarrow 0110, \ b \rightarrow 1001, \ c \rightarrow 1.$$
\end{example}

An interesting relation between the palindromic and factor complexity has been revealed in~\cite{BaMaPe}.
\begin{thm}\label{Balazi_fac_pal}
Let $\u$ be an infinite word with the language $\L$ closed under reversal. Then
\begin{equation}\label{fac_pal}
\P(n+1)+\P(n)\leq \Delta \C(n)+2 \quad \text{for all $n \in \N$}.
\end{equation}
\end{thm}
In fact, the above relation is stated in~\cite{BaMaPe} for
uniformly recurrent words, however the proof requires only recurrent words.
Theorem~\ref{Balazi_fac_pal} implies that infinite words reaching the equality in~\eqref{fac_pal} are in a~certain sense
opulent in palindromes. Another measure of opulence in palindromes has been provided in~\cite{DrJuPi}.
\begin{thm}\label{max_w+1}
Every finite word $w$ contains at most $|w|+1$ palindromes (including the empty word).
\end{thm}
\begin{defi}\label{rich}
An infinite word $\u$ satisfying that every factor $w$ of $\u$ contains
$|w|+1$ palindromes is called rich in palindromes.
\end{defi}

The following equivalent definitions of richness have been proved in~\cite{GlJuWiZa}, \cite{BuLuGlZa}, \cite{BuLuGlZa2}, respectively.
\begin{thm}\label{equiv_rich}
For any infinite word $\u$ the following conditions are equivalent:
\begin{enumerate}
\item $\u$ is rich,
\item any return word of a~palindromic factor of $\u$ is a~palindrome,
\item for any factor $w$ of $\u$, every factor of $\u$ that contains $w$ only as its prefix and $\overline w$ only as its suffix is a~palindrome,
\item each factor of $\u$ is uniquely determined by its longest palindromic prefix and its longest palindromic suffix.
\end{enumerate}
\end{thm}
We will need for our further purposes an implication that holds only for languages closed under reversal.

\begin{coro}[\cite{BuLuGlZa}]
\label{CuRequiv_rich}
Let $\u$ be a~rich infinite word with the language closed under reversal.
Then for any factor $w$ of $\u$, the occurrences of $w$ and $\overline{w}$ alternate.
\end{coro}

A~natural question is whether infinite words reaching the equality in~\eqref{fac_pal} coincide with rich words.
The following theorem proved in~\cite{BuLuGlZa} for uniformly recurrent words,
however valid even for infinite words with the language closed under reversal, provides an answer.
\begin{thm}\label{rich_opulent}
Let $\u$ be an infinite word with the language $\L$ closed under reversal.
Then $\u$ is rich if and only if $\P(n+1)+\P(n)=\Delta \C(n)+2$ for all $n \in \mathbb N$.
\end{thm}

Let us explain that Theorem~\ref{rich_opulent} is slightly stronger than the equivalence of richness
and the equality in~\eqref{fac_pal} for uniformly recurrent words, proved in~\cite{BuLuGlZa}. In other words, the following statement is a~corollary of Theorem~\ref{rich_opulent}.
\begin{corollary}
Let $\u$ be a~uniformly recurrent infinite word. Then $\u$ is rich if and only if  $\P(n+1)+\P(n)=\Delta \C(n)+2$ for all $n \in \mathbb N$.
\end{corollary}
\begin{proof}
If $\L$ is closed under reversal, then the statement follows from Theorem~\ref{rich_opulent}.
If $\L$ is not closed under reversal, then by Theorem~\ref{unif_rec_pal}, $\u$ contains only a~finite number of palindromes. It is then readily seen that $\u$ is neither rich, nor the equality in~\eqref{fac_pal} is attained for all $n \in \N$.
\end{proof}
Let us correct the following example given in~\cite{BuLuGlZa}. The word $\u$ generated by the substitution $a \to aba, \ b \to bb$ is recurrent, however not uniformly recurrent, and the language $\L$ is closed under reversal. By inspection of the complete return words of palindromic factors, applying Theorems~\ref{equiv_rich} and~\ref{rich_opulent}, it may be proved that the equality in~\eqref{fac_pal} is attained. The authors of~\cite{BuLuGlZa} claimed that ${\mathcal P}(2)+{\mathcal P}(3)\not = \Delta{\mathcal C}(2)+2$. This mistake is however based on the fact that ${\mathcal C}(3)=5$ and not $6$.

Let us mention as an open problem the following question.
``Does the equivalence of richness and the equality in~\eqref{fac_pal} hold for
a~larger class than words with the language closed under reversal? For instance for all recurrent words?''

The following observations may serve as hints:
\begin{itemize}
\item
It does not hold for non-recurrent infinite words in general.
The infinite word $ab^\omega$ is given in~\cite{BuLuGlZa} as an example of a~rich non-recurrent infinite word (with the language of course not closed under reversal),
which does not reach the equality in~\eqref{fac_pal} for all $n \in \N$.
\item Notice that both rich infinite words and infinite words reaching the equality in~\eqref{fac_pal} contain infinitely many palindromes.
\item If $\u$ is rich and recurrent, then $\L$ is closed under reversal (proved in~\cite{GlJuWiZa}, Proposition 2.11).
\end{itemize}

The rest of this section is devoted to the relation between richness and bilateral orders of factors.
The following proposition reveals some information on bilateral orders of palindromic bispecial factors in an infinite word with the language closed under reversal.

\begin{proposition}
\label{prop_more_than_1}
Let $\u$ be an infinite word whose language is closed under reversal.
Then the bilateral order $\b(w)$ of a palindromic bispecial factor $w \in \L$
has a~different parity than the number of palindromic extensions of $w$.
\end{proposition}

\begin{proof}
Let $w$ be a~palindromic BS factor of $\u$. 
On one hand, as the language is closed under reversal, we have $\# \Lext(w)  = \# \Rext(w)$.
Consequently, from the definition of bilateral order one can see that the parity of $\# \Bext(w)$
is different from the parity of $\b(w)$. On the other hand, the parity of the number of palindromic extensions of $w$ equals the parity of $\# \Bext(w)$ since for any $a,b \in \A$, if $awb \in \L$, then $bwa \in \L$.
\end{proof}

In the sequel, we will state and prove a~new equivalent definition of rich words by means of bilateral orders.

\begin{thm}\label{thm:minus1}
Let $\u$ be an infinite word with the language $\L$ closed under reversal. Then $\u$ is rich if and only if any bispecial factor $w$ of $\u$ satisfies:
\begin{itemize}
\item if $w$ is non-palindromic, then
$$\b(w)= 0,$$
\item if $w$ is a~palindrome, then
$$\b(w)= \# \Pext(w) - 1.$$
\end{itemize}
\end{thm}

The following lemma will provide the most important tool for the proof of Theorem~\ref{thm:minus1}.

\begin{lem}\label{lemma:minus1}
Let $\u$ be a~rich infinite word whose language is closed under reversal.
Then it holds for any bispecial factor $w$:
\begin{itemize}
\item if $w$ is non-palindromic, then
$$\b(w)\geq 0,$$
\item if $w$ is a~palindrome, then
$$\b(w)\geq \# \Pext(w) - 1.$$
\end{itemize}
\end{lem}

\begin{proof}

Let $w$ be a non-palindromic BS factor.
By the definition of $\b(w)$, we want to prove
$$
\# \Bext(w) \geq \# \Rext(w) + \# \Lext(w) - 1.
$$
We will construct a bipartite oriented graph $G$ having its set of vertices $V$ defined as
$$
V = \left \{ wa | a \in \Rext(w) \right \} \cup \left \{ \overline{w}a | a \in \Rext(\overline{w}) \right \}.
$$
There is an oriented edge from $wa$ to $\overline{w}b$ if there exists a factor $vb \in \L$ such that
 $wa$ is its prefix, $\overline{w}b$ is its suffix and factors $w$ and $\overline{w}$ occur each exactly once in $vb$.
 Furthermore, there is an oriented edge for $\overline{w}x$ to $wy$ if there exists a factor $vy \in \L$ such that
 $\overline{w}x$ is its prefix, $wy$ is its suffix and factors $w$ and $\overline{w}$ occur each exactly once in $v$.

\begin{figure}[ht]
\begin{center}
\begin{picture}(250,150)

\put(50,20){\circle*{5}}
\put(200,20){\circle*{5}}
\put(60,20){\vector(1,0){130}}

\put(48,8){$\overline{w}a$}
\put(198,8){$wb$}

\put(48,57){\rule{159pt}{0.5pt}}
\put(33,50){$\ldots$}
\put(48,50){$\! \mid \underbrace{\overline{w} \mid a \hspace{110pt} w}_{\displaystyle v} \mid b \mid$}
\put(211,50){$\ldots$}
\put(48,47){\rule{159pt}{0.5pt}}

\put(50,100){\circle*{5}}
\put(200,100){\circle*{5}}
\put(60,100){\vector(1,0){130}}

\put(48,88){$wa$}
\put(198,88){$\overline{w}b$}

\put(48,137){\rule{159pt}{0.5pt}}
\put(33,130){$\ldots$}
\put(48,130){$\! \mid \underbrace{w \mid a \hspace{110pt} \overline{w}}_{\displaystyle v} \mid b \mid$}
\put(211,130){$\ldots$}
\put(48,127){\rule{159pt}{0.5pt}}

\end{picture}
\end{center}
\caption{Incidence relation in the graph $G$.}
\label{inc_G}
\end{figure}
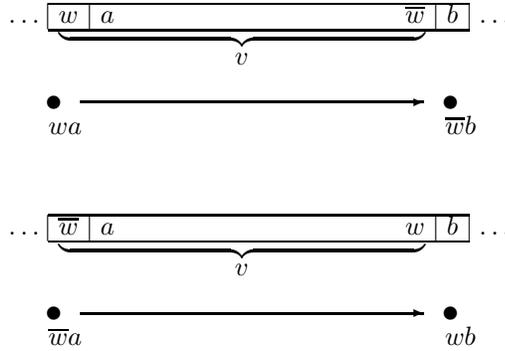
Due to Theorem~\ref{equiv_rich}, such a factor $v$ is a~palindrome.
 Therefore the existence of an edge from $wa$ to $\overline{w}b$ implies $a\overline{w}b \in \L$, and so $bwa \in \L$, too.
 Analogously, if there is an edge from $\overline{w}x$ to $wy$, we have $xwy \in \L$.

 By \Cref{CuRequiv_rich}, the occurrences of $w$ and $\overline{w}$ alternate. Thus, to any factor of $\u$ corresponds a~path in $G$.
 As $\u$ is recurrent, the graph $G$ is strongly connected.

 As a~consequence, the number of pairs of its vertices which are connected by an edge is greater than or equal to the number of its vertices minus~$1$.
 We have
 $$
 \# \Bext(w) \geq \# \Rext(w) + \# \Rext(\overline{w}) - 1.
 $$
Since $\Rext(\overline{w}) = \Lext(w)$ the proof of the first part is finished.

\vspace{\baselineskip}

Let $w$ be a~palindromic BS factor.
Let us consider this time a~graph $G$ whose set of factors $V$ is defined as
$$
V = \left \{ wa | a \in \Rext(w) \right \}.
$$
There is an edge from $wa$ to $wb$ if there exists a factor $vb \in \L$ such that $v$ is a complete return word to $w$ that has $wa$ as a~prefix.
As $\u$ is rich, $v$ is a~palindrome.
Due to the recurrence of $\u$, for every $awb \in \L$, $a \neq b$, there exists an edge in $G$ going from $wa$ to $wb$.
As the language is closed under reversal, the edge going from $wb$ to $wa$ is in $G$, too.
Therefore
$$
\# \left \{ awb \in \L | a \neq b \right \} = 2 \times \text{ the number of pairs of distinct vertices connected by an edge}.
$$
Owing to the recurrence of $\u$, the graph $G$ is strongly connected, thus the number of pairs of distinct vertices connected by an edge is greater than or equal to the number of vertices of $G$ minus $1$,
which equals $\# \Rext(w) - 1$.
We find
$$
\# \Bext(w) = \# \left \{ awb \in \L | a \neq b \right \} + \# \Pext(w) \geq 2 \left( \# \Rext(w) - 1 \right) + \#\Pext(w).
$$
As $\Rext(w) = \Lext(w)$, the statement is proved.

\end{proof}
\begin{proof}[Proof of Theorem~\ref{thm:minus1}]
\noindent $(\Leftarrow)$: Let us show by mathematical induction that
$$\Delta \C(n)+2=\P(n+1)+\P(n) \quad \text{for all $n \in \mathbb N$}.$$
Since $\L$ is closed under reversal, this means by Theorem~\ref{rich_opulent} that $\u$ is rich.

The assumption on bilateral orders and the fact that non-bispecial palindromic factors have a~unique palindromic extension guarantee the following equality for all $n \in \mathbb N$:
\begin{equation}\label{sec_diff_pal}
\Delta^2 \C(n)=\sum_{w \in {\mathcal L}_n(\u)}\b(w)=\sum_{\substack{w \in {\mathcal L}_n(\u)\\ w=\overline{w}}}\left(\#\Pext(w)-1\right)=\P(n+2)-\P(n).
\end{equation}
\noindent For $n=0$, we can write $\Delta \C(0) + 2 = \C(1) - \C(0) + 2 = \# \A + 1$.
On the other hand we have $\P(1) + \P(0) = \# \A + 1$. \\
\noindent Take $N \in \mathbb N$. Assume $\Delta \C(n)+2=\P(n+1)+\P(n)$ holds for all $n <N$.
Using the induction assumption and \eqref{sec_diff_pal}, we obtain
$$\begin{array}{rcl}
\Delta \C(N)+2&=&\left(\Delta \C(N)-\Delta\C(N-1)\right)+\left(\Delta \C(N-1)+2\right)\\
&=&\Delta^2\C(N-1)+\left(\P(N-1)+\P(N)\right)\\
&=&\left(\P(N+1)-\P(N-1)\right)+\left(\P(N-1)+\P(N)\right)\\
&=&\P(N+1)+\P(N).
\end{array}$$
\noindent $(\Rightarrow)$: Take $n \in \mathbb N$ arbitrary. We will prove the statement of the theorem for all BS factors of length $n$.

As $\u$ is rich and the language $\L$ is closed under reversal, we have by Theorem~\ref{rich_opulent}
$$\Delta \C(k)+2=\P(k+1)+\P(k) \quad \text{for all $k \in \mathbb N$}.$$
Applying this equality, we will deduce the form of $\Delta^2\C(n)$.
$$\Delta^2\C(n)=\left(\Delta\C(n+1)+2\right)-\left(\Delta\C(n)+2\right)=\left(\P(n+2)+\P(n+1)\right)-\left(\P(n+1)+\P(n)\right)=\P(n+2)-\P(n).$$
Consequently, we obtain
$$\sum_{w \in {\mathcal L}_n(\u)}\b(w)=\Delta^2\C(n)=\P(n+2)-\P(n)=\sum_{\substack{w \in {\mathcal L}_n(\u)\\ w=\overline{w}}}\left(\#\Pext(w)-1\right).$$
Palindromic factors that are not BS have obviously exactly one palindromic extension. Thus, we can rewrite the previous equality
\begin{equation}\label{onlyBS}
\sum_{w \in {\mathcal L}_n(\u)}\b(w)=\sum_{\substack{w \in {\mathcal L}_n(\u)\\ w=\overline{w}, w \ \text{BS}}}\left(\#\Pext(w)-1\right).
\end{equation}
Let us split the sum of bilateral orders into two parts and use Lemma~\ref{lemma:minus1}

\begin{equation}\label{BSsplit}
\sum_{w \in {\mathcal L}_n(\u)}\b(w)=\sum_{\substack{w \in {\mathcal L}_n(\u)\\ w\not=\overline{w}, \ w \ \text{BS}}}\b(w)+\sum_{\substack{w \in {\mathcal L}_n(\u)\\ w=\overline{w}, \ w \ \text{BS}}}\b(w)
\geq  \sum_{\substack{w \in {\mathcal L}_n(\u)\\ w\not=\overline{w},\ w \ \text{BS}}}\b(w)+\sum_{\substack{w \in {\mathcal L}_n(\u)\\ w=\overline{w},\ w \ \text{BS}}}\left(\#\Pext(w)-1\right).
\end{equation}
This in combination with \eqref{onlyBS} gives $\displaystyle \sum_{\substack{w \in {\mathcal L}_n(\u)\\ w\not=\overline{w}, \ w \ \text{BS}}}\b(w)=0$. By Lemma~\ref{lemma:minus1}, bilateral orders of such factors are non-negative, which implies $\b(w)=0$ for all non-palindromic BS factors.
Since the equality is reached in~\eqref{BSsplit}, we obtain $\displaystyle \sum_{\substack{w \in {\mathcal L}_n(\u)\\ w=\overline{w},\ w \ \text{BS}}}\b(w)=
\sum_{\substack{w \in {\mathcal L}_n(\u)\\ w=\overline{w}, \ w \ \text{BS}}}\left(\#\Pext(w)-1\right).$
Together with Lemma~\ref{lemma:minus1}, this results in $\b(w)=\#\Pext(w)-1$ for all palindromic BS factors.
\end{proof}
\section{Equivalent definitions of Sturmian words} \label{sec:equiv}
Let us stress a~close link between periodicity and complexity (revealed by Hedlund and Morse~\cite{HeMo1}).
On one hand, the complexity of eventually periodic words is bounded. On the other hand, if there exists $n \in
\mathbb N$ such that ${\mathcal C}(n) \leq n$, then the complexity is bounded and the infinite
word $\u$ is eventually periodic. In consequence, the complexity of aperiodic words
satisfies ${\mathcal C}(n) \geq n+1$ for all $n \in \mathbb
N$. {\em Sturmian words} are defined as infinite words with the complexity ${\mathcal
C}(n)=n+1$ for all $n \in \mathbb N$.
This condition on complexity implies many properties. Let us list some of them.
If $\u$ is a~Sturmian word, then $\u$ has the following properties:
\begin{itemize}
\item $\u$ is a~binary word,
\item $\u$ is aperiodic,
\item the language $\L$ is closed under reversal,
\item the language $\L$ contains infinitely many palindromes,
\item the word $\u$ is uniformly recurrent,
\item the language $\L$ contains no weak bispecial factors,
\item $\u$ is rich.
\end{itemize}

There exist many equivalent definitions of Sturmian words. The following
theorem summarizes several of their well-known combinatorial characterizations.
\begin{thm}
\label{eq_sturm}
Let $\u$ be an infinite word over the alphabet $\A$.
The properties listed below are equivalent:
\begin{enumerate}[(i)]
\item $\u$ is Sturmian, i.e., $\FC{n} = n + 1$ for all $n$,
\item $\u$ is binary and contains a~unique left special factor of every length,
\item $\u$ is binary, aperiodic and every bispecial factor is ordinary,
\item any factor of $\u$ has exactly two return words,
\item $\u$ contains one palindrome of every even length and two
palindromes of every odd length,
\item $\u$ is binary and every palindrome has a~unique palindromic extension,
\item $\u$ is aperiodic and balanced,
\item $\u$ is aperiodic and ${\mathcal AC}(n)=2$ for all $n \in \mathbb N, \ n \geq 1$.
\end{enumerate}
\end{thm}
The characterization by return words is due to Vuillon~\cite{Vu} and the one by the abelian complexity is a~consequence of the works by Coven and Hedlund~\cite{CoHe}.
The equivalent definition based on the balance property comes already from Hedlund and Morse~\cite{HeMo}.
The two equivalent properties concerning palindromes have been proved by Droubay and Pirillo~\cite{DrPi}.
Notice that the sixth property can be equivalently rewritten as
$${\mathcal P}(n)+{\mathcal P}(n+1)=3 \quad \text{for all $n \in \mathbb N$,}$$
and also as
$$\P(n+2)=\P(n) \quad \text{for all $n \in \N$}.$$
Let us recall that $\P(0)=1$ since the empty word is considered to be a~palindrome.
\section{Generalizations of Sturmian words} \label{generalizedSturm}
We have seen that Sturmian words can be defined in many equivalent ways.
As a~matter of course,
various generalizations to multiliteral alphabets have been
suggested and studied.
\subsection{Two well-known generalizations}\label{WellKnown}
The most studied generalizations are Arnoux-Rauzy words and words coding $k$-interval exchange
transformation.

{\em Arnoux-Rauzy words} (or {\em AR words} for simplicity) are
infinite words with the language closed under reversal and containing exactly one LS
factor $w$ of every length,
and such that every LS factor has the same number $k$ of left extensions, i.e., $\# \Lext(w)=k$.
Their alphabet $\A$ has $k$ letters since the empty word has exactly $k$ left extensions.
AR words are aperiodic and satisfy $\C(n)=(k-1)n+1$ for all $n \in \N$.
They have been defined and studied in~\cite{DrJuPi}, the following properties have been proved ibidem.
The language of AR words contains infinitely many palindromes,
they are uniformly recurrent, rich, and have only ordinary BS factors.
AR words form a~subclass of extensively studied {\em episturmian
words} (see for instance \cite{GlJu}), defined as infinite words that have the language closed
under reversal and contain at most one LS factor of every length.

\vspace{0.5cm}
Another well-known generalization of Sturmian words
is provided by {\em words coding $k$-interval exchange transformation}.
Let us state their definition and then explain why such words generalize
Sturmian words to $k$-letter alphabets.
Take positive numbers $\alpha_1, \dots,
\alpha_k$ such that
$\sum_{i=1}^k \alpha_i=1$. They define a~partition of the interval
$I = [0, 1)$ into $k$ subintervals $$I_j=
\bigl[\sum_{i=1}^{j-1}\alpha_i, \sum_{i=1}^j \alpha_i\bigr), \  j =
1, 2,\dots, k.$$
The {\em interval exchange transformation} is a~bijection $T: I \to I$
given by the prescription
$$T(x)=x+c_j \quad \text{for all $x \in I_j, \ j \in \{1,2,\dots,k\}$},$$
where $c_j$ are suitably chosen constants.
Since $T$ is a~bijection, the intervals $T(I_1), T(I_2), \dots, T(I_k)$ form a~partition of $I$.
The orders of $T(I_j)$ in the partition define a~permutation
$\pi:\{1,2,\dots,k\}\to \{1,2,\dots,k\}$ and this permutation $\pi$ determines uniquely the constants $c_j$.
For instance, if the permutation $\pi$ is symmetric, i.e., $\pi=\left(\begin{smallmatrix}1&2&\dots &k-1&k\\ k&k-1&\dots&2&1\end{smallmatrix}\right)$,
then the transformation $T$ is of the following form
$$T(x) = x
+\sum_{i>j}\alpha_i - \sum_{i<j}\alpha_i \quad \mbox{for}\quad x \in
I_j.$$

The infinite word $\u=u_0u_1u_2\dots$ over ${\A}=\{a_1, \dots, a_k\}$ associated with $T$ is defined as
$$u_n:=a_j \quad \mbox{if $\quad T^{n}(x) \in I_j$}$$ and is called a~{\em word coding
$k$-interval exchange transformation} ({\em $k$-iet
word} for short).

From the point of view of combinatorics on words, an important role is played by those transformations whose orbit
for an arbitrary $x \in I$ is dense in $I$, i.e., the closure of $\{T^n(x)\bigm | n \in \N\}$ is the whole interval~$I$.
A~sufficient condition for this property represents the so-called i.d.o.c. (consult~\cite{Ke}) and the irreducibility of the permutation $\pi$.
In the sequel, let us assume that $T$ satisfies both of these properties.
The $k$-iet word is then uniformly recurrent, its language does not depend on the position of
the starting point $x$, but only on the transformation~$T$, its complexity satisfies $\C(n)=(k-1)n+1$ for all $n \in \N$ and no BS factor is weak.

The language of the $k$-iet word $\u$ is closed under reversal if and only if the permutation $\pi$ is symmetric.
In such a~case, the language $\L$ contains infinitely many palindromes and, as shown in~\cite{BaMaPe},
the equality in~\eqref{fac_pal} is attained. Hence, according to Theorem~\ref{rich_opulent},
the $k$-iet words are rich. It is easy to describe the infinite palindromic branches for such $k$-iet words.
The one with the empty word as its center is obtained as the coding of the orbit $\{T^n(x)| n \in \mathbb Z\}$ with the starting point $x=1/2$ and
the branch with the center $a_j \in \A$ as the coding of the orbit with the starting point $x=\sum_{i < j}\alpha_i+\alpha_j/2$.

The $k$-iet words provide a~generalization of Sturmian words due to the well-known connection between Sturmian and mechanical words~\cite{Lo2}.

\begin{thm}Let $\u$ be an infinite word.
Then $\u$ is Sturmian if and only if $\u$ is a~2-iet word with an irrational partition of the unit interval.
\end{thm}

Recently, in \cite{SmUl}, a different generalization of Sturmian sequences is considered.
It in fact corresponds to~a special subclass of $k$-iet words given by coding a trajectory in a regular $2n$-gon.


\subsection{Combinatorial generalizations}
Let us write down and baptize the generalizations of properties from Theorem~\ref{eq_sturm}.
We will then refer to them and study their relations.
Let $\u$ be an infinite word over the alphabet ${\A}$. Denote $k=\#{\A}$.
\begin{enumerate}
\item{Property $\mathcal C$}:

the factor complexity of $\u$ satisfies
${\mathcal C}(n)=(k-1)n+1$ for all $n \in \mathbb N$.
\item{Property $\LR$}:

$\u$ contains one left special and one right special factor of every length.
\item {Property $\BO$}:

all bispecial factors of $\u$ are ordinary.

\item{Property $\Ret$}:

any factor of $\u$ has exactly $k$ return words.
\item{Property $\P$}:

the palindromic complexity of $\u$ satisfies
${\P}(n)+{\P}(n+1)=k+1$ for all $n \in \mathbb N$.
\item{Property $\PE$}:

every palindrome has a~unique palindromic extension in $\u$.
\item{Balance properties}:
\begin{enumerate}
\item{Property $\B_{\forall}$}:

$\u$ is aperiodic and for all $a \in {\A}$ and for all factors $w,v \in \L$ with $|w|=|v|$ it holds
$$||w|_a-|v|_a|\leq k-1.$$
\item{Property $\B_{\exists}$}:

$\u$ is aperiodic and there exists $a \in {\A}$ such that for all factors $w,v \in \L$ with $|w|=|v|$ it holds
$$||w|_a-|v|_a|\leq k-1.$$
\item{Property $\AC$}:

$\u$ is aperiodic and the abelian complexity of $\u$ satisfies $\AC(n)=k$ for all $n \in \N, \ n \geq 1$.
\end{enumerate}
\end{enumerate}

At first, let us mention which properties are satisfied by the two generalizations of Sturmian words from Section~\ref{WellKnown}. AR words fulfill Properties: $\C, \LR, \BO, \Ret, \P, \PE$ and $k$-iet words satisfy Properties:
$\C, \BO, \Ret$. If moreover the permutation defining the $k$-iet word is symmetric, then these words have Properties
$\P$ and $\PE$. Property $\LR$ does not hold for $k$-iet words.

It follows directly from the definition that some Properties imply others. For instance, by~\eqref{second_diff} $\BO$ implies $\C$. They are not equivalent as shown by the following example taken from~\cite{Fer}.
\begin{example}[$\C \not \Rightarrow \BO$] \label{ex:recoded_chacon}
The infinite ternary word $\lim_{n \to \infty}\varphi^n(a)$, where
$\varphi(a)=ab, \ \varphi(b)=cab, \ \varphi(c)=ccab$ -- a~recoding of the Chacon substitution -- has the complexity $2n+1$ for every $n \in \N$,
but contains infinitely many strong and weak BS factors.
\end{example}
In the sequel, we will show that no two of these properties are equivalent on a~multiliteral alphabet.

Concerning Properties $\B_{\forall}, \B_{\exists}$ and $\AC$, we will not treat them but in the last section since they are very restrictive, and consequently, satisfied only by a~small class of infinite words.

\subsection{Property $\LR$}

Property $\LR$ does not characterize AR words since it is satisfied by a~larger class of words. Infinite words with the language closed under reversal and satisfying Property $\LR$ coincide with extensively studied aperiodic episturmian words. Nevertheless, Property $\LR$ may be satisfied by words whose language is not closed under reversal, as illustrated in~\cite{DrJuPi} by the following example. It shows also that Property $\LR$ does not guarantee Properties $\C, \BO, \Ret, \P, \PE$.
\begin{example}[$\LR \not \Rightarrow$ closeness under reversal, $\C, \BO, \Ret, \P, \PE $] \label{ex:lrlr}
If we construct an infinite word $\u$ so that we replace $b$ with $bc$ in the Fibonacci word $abaababaabaabab\dots$,
 the fixed point of $\varphi: a \to ab, \ b \to a$,
 then $bc$ is a~factor of $\L$, however $cb$ not.
 It is easy to see that such a~word has still a~unique infinite RS and a~unique LS branch (the infinite word $\u$ itself).
 Consequently, Property $\LR$ is preserved.
 However, both of these infinite special branches have only two extensions,
 hence Property $\C$ (and $\BO$ as well) fails.
 The factor $c$ has only two return words $caab$ and $cab$,
 hence Property $\Ret$ does not hold.
 Moreover, as $\u$ is uniformly recurrent and its language is not closed under reversal,
 it contains by Theorem~\ref{unif_rec_pal} only a~finite number of palindromes.
 Therefore, Properties $\P$ and $\PE$ are not satisfied.
\end{example}

On the other hand, observing $k$-iet words, we learn that none of Properties $\C, \BO, \Ret, \P, \PE$ imply $\LR$.
The problem to describe the class of infinite words with Property $\LR$ whose language is not closed under reversal requires a~further study.

\subsection{Property $\Ret$}
Let us recall that infinite words with Property $\Ret$ are necessarily uniformly recurrent.
If their language is not closed under reversal, then it cannot contain infinitely many palindromes by Theorem~\ref{unif_rec_pal}. Such words exist, as illustrated by the following example, therefore, Property $\Ret$ does not imply $\P$.
\begin{example}[$\Ret \not \Rightarrow \P$] \label{ex:bapest}
The fixed point $\u$ of $\varphi$, where
$\varphi(a)=aab, \ \varphi(b)=ac, \ \varphi(c)=a$, contains $bac$, but $cab$ is not its factor. The fact that
every factor of $\u$ has three return words is explained in~\cite{BaPeSt}
for a~whole class of infinite words coding $\beta$-integers.
\end{example}

We have seen that AR words and $k$-iet words have both Property $\Ret$ and $\C$, however, as shown in~\cite{Fer} by the following example, Property $\C$ does not imply Property $\Ret$ on multiliteral alphabets.
\begin{example}[$\C \not \Rightarrow \Ret$] \label{ex:ter_not_R}
The fixed point of $\varphi: a \to ab, \ b \to cab, c \to ccab$ -- the above mentioned recoding of the Chacon substitution -- has the complexity $2n+1$ for every $n \in \N$,
but contains more than three return words of certain factors
(for example the factor $bc$ has $4$ return words:  $bca$, $bcca$, $bcaba$ and $bccaba$.
\end{example}
The following theorems come from the paper~\cite{BaPeSt} that is devoted to the study of Property $\Ret$ for infinite words on multiliteral alphabets.
Let us observe once more AR words and $k$-iet words, these classes satisfy  not only Property $\C$, but also Property $\BO$. It is thus natural to ask whether Property $\BO$ guarantees $\Ret$. The corollary of the following theorem will provide an answer.
\begin{thm}\label{BOImplRet}
If $\u$ is an infinite word with no weak BS factors, then $\u$ has Property
$\Ret$ if and only if $\u$ is uniformly recurrent and satisfies $\C$.
\end{thm}
Let us underline, an infinite word $\u$ has Property $\BO$ if and only if it has Property $\C$ and contains no weak BS factors.
It results in the advertised corollary.
\begin{coro}
Let $\u$ be a~uniformly recurrent infinite word. Then
$$\BO \Rightarrow \Ret.$$
\end{coro}
If we restrict our consideration to the ternary alphabet, the implication can be reversed.
\begin{thm}\label{RetImplBO}
Let $\u$ be a~ternary uniformly recurrent infinite word.
Then
$$\BO \Leftrightarrow \Ret.$$
\end{thm}

As soon as the alphabet has more than three letters,
Property $\Ret$ does not imply Property $\BO$ any more.
\begin{example}[$\Ret \not \Rightarrow \BO$] \label{ex:ter_R_not_C}
The uniformly recurrent infinite word
$\u=\lim_{n \to \infty}\varphi^n(a)$, where
$$
\varphi(a)=acbca, \ \varphi(b)=acbcadbdaca, \ \varphi(c)=dbcbdacadbd, \ \varphi(d)=dbcbd,
$$
satisfies $\Ret$, but not $\C$ (since $\mathcal C(n)$ is even for all $n \in \mathbb N$) and $\u$ contains, of course, weak BS factors. For details consult~\cite{BaPeSt}.
\end{example}
The question whether there exists a~nice characterization of words with Property $\Ret$ on alphabets with more than three letters remains open.
\subsection{Property ${\P}$ and $\PE$}
The paper~\cite{BaPeSta} is focused on the study of Properties ${\P}$ and $\PE$.
As soon as an infinite word $\u$ has Property $\PE$, then $\u$ has exactly one infinite palindromic branch with center $a$ for every letter $a \in \A$ and one infinite palindromic branch with center $\varepsilon$. Therefore, $\u$ contains exactly $\#\A$ palindromes for every odd length (central factors of palindromic branches with centers $a \in \A$) and one palindrome for every even length (central factor of the infinite palindromic branch with center $\varepsilon$). Consequently, Property $\P$ is also satisfied by $\u$.

Let us recall that Property $\P$ may be reformulated in the following way
\begin{equation}
\P(n+2)=\P(n) \quad \text{for all $n \in \N$,}
\end{equation}
where $\P(0)=1.$
We will equally use both of the forms of Property $\P$.

Let $\u$ be an infinite word satisfying $\PE$. The language $\L$ contains infinitely many palindromes, but it need not be closed under reversal,
neither recurrent nor rich as illustrated by the following example.
\begin{example}[$\PE \not \Rightarrow$ closeness under reversal, $\PE \not \Rightarrow$ richness] \label{ex:inf_pals_not_rec}
The infinite word $\u$ on the alphabet $\{a,b,c\}$ defined in the following way:
$$\u=c\ a\underbrace{\ c \ c \ }_{2\times}b\underbrace{ccc}_{3\times}a\underbrace{cccc}_{4\times}b\underbrace{ccccc}_{5\times}
a\underbrace{cccccc}_{6\times}b\underbrace{ccccccc}_{7\times}a\dots$$
has three infinite palindromic branches with centers $a, b$ and $c$
$$\dots cccaccc\dots, \quad \dots cccbccc\dots, \quad \dots ccccccc\dots$$
and one infinite palindromic branch with central factors of even length of the form $\dots cccccccc \dots$
Indeed, $\u$ has the factor $accb$, however its mirror image $bcca$ does not belong to the language $\L$.
Moreover, $\u$ is not rich since the prefix $caccbccca$ of length $9$ contains only $9$ palindromes: \\
$\varepsilon$, $a$, $b$, $c$, $cc$, $cac$, $cbc$, $ccc$ and $ccbcc$.
\end{example}
However, if the language $\L$ is closed under reversal, then it is possible to say more about the relation of Properties $\P$ and $\C$ and the richness of $\u$. When both $\P$ and $\C$ are satisfied, the equality in~\eqref{fac_pal} is reached. Application of Theorem~\ref{rich_opulent} provides us with the following corollary.
\begin{coro}\label{PCrich}
Let $\u$ be an infinite word whose language is closed under reversal. Then
$$\P + \C \Rightarrow \text{richness of $\u$}.$$
\end{coro}
The first example shows that Property $\P$ itself does not guarantee richness even if the language is closed under reversal. The second one illustrates that the implication in~Corollary~\ref{PCrich} cannot be reversed.
\begin{example}[$\PE \not \Rightarrow $ richness, $\PE \not \Rightarrow \C$] \label{ex:bo}
A known example of an infinite word with the language closed under reversal and with a~higher factor
complexity is the billiard sequence on three letters, for which $\C(n) = n^2 + n + 1$.
As shown in~\cite{Bo}, such words satisfy Property $\PE$, hence $\P$ as well.
Consequently, billiard sequences do not reach the upper bound in~\eqref{fac_pal} and by Theorem~\ref{rich_opulent} cannot be rich.
\end{example}
\begin{example}[richness $\not \Rightarrow \P$, richness $\not \Rightarrow \C$] \label{ex:rovnost_bez_C}
Let $\varphi$ be defined on an $m$-letter alphabet as follows:
$$
\varphi(0) = 0^t1, \quad \varphi(1) = 0^t2, \quad \ldots, \varphi(m-2) = 0^t(m-1), \quad \varphi(m-1) = 0^s,
$$
where $s,t \in \N$ and $t \geq s \geq 2$.
The fixed point $\u$ of $\varphi$ satisfies the equality $\P(n+1)+\P(n)=\Delta \C(n)+2$ for all $n$. As the language is closed under reversal, by Theorem~\ref{rich_opulent} $\u$ is rich.
Property $\P$ is not satisfied since the sum $\P(n+1) + \P(n)$ is not constant.
Further properties of palindromes in $\u$ can be found in~\cite{AmMaPeFr}.
\end{example}

Let us examine in the sequel the connection between Properties $\C$ and $\P$, resp. $\C$ and $\PE$.
\subsubsection{Ternary alphabet}
Let us limit our considerations to the ternary alphabet. The following theorem and examples come from~\cite{BaPeSta}.
\begin{thm}\label{PandPE}
Let $\u$ be an infinite ternary word with the language closed under reversal. Then
\begin{enumerate}
\item
$\mathcal C \Rightarrow \P,$
\item $\BO \Rightarrow \PE.$
\end{enumerate}
\end{thm}
The implication in Theorem~\ref{PandPE} cannot be reversed. We have already illustrated in Example~\ref{ex:bo} that even the stronger property $\PE$ does not ensure $\C$.
Let us provide one more counterexample - a~fixed point of a~substitution.
\begin{example}[$\PE \not \Rightarrow \C$] \label{ex:ter_PE_not_C}
Denote by $\u$ the infinite ternary word being the fixed point of the substitution $\Phi$ defined by
\begin{equation}\label{U}
\Phi(a)=aba, \quad \Phi(b)=cac, \quad \Phi(c)=aca.
\end{equation}
Then the language of $\u$ is closed under reversal. On one hand, $\u$ has Property ${\mathcal {PE}}$, consequently, $\u$ has Property $\P$, too.
On the other hand, Property $\C$ fails and $\L$ contains infinitely many weak BS factors.
\end{example}

Properties $\P$ and $\PE$ are equivalent for binary words. However already for ternary words, the implication $\P \Rightarrow \PE$ does not hold any more.
\begin{example}[$\P \not \Rightarrow \PE$] \label{ex:ter_P_not_PE}
Let $\mathbf v$ be the ternary infinite word defined by ${\mathbf v}=\Psi(\u)$, where
$\Psi:\{A,B\}^{*}\to \{a,b,c\}^{*}$ is the morphism given by
\begin{equation*}\label{Psi}
\Psi(A)=bc \quad \text{and} \quad \Psi(B)=baa,
\end{equation*}
and $\u$ is the fixed point of the substitution $\varphi$
defined by
\begin{equation*}\label{varphi}
\varphi(A)=ABBABBA, \quad \varphi(B)=ABA.
\end{equation*}
Then $\mathbf v$ satisfies ${\P}$, but does not satisfy $\PE$.
\end{example}
The relation between $\Ret$ and $\P$ follows from Theorem~\ref{PandPE} and Theorem~\ref{RetImplBO}.
\begin{coro}
Let $\u$ be an infinite ternary word with the language closed under reversal. Then  $$\mathcal R \Rightarrow \PE.$$
\end{coro}
The implication cannot be reversed.
\begin{example}[$\PE \not \Rightarrow \Ret$] \label{ex:no_Ret}
Consider the fixed point $\u$ of the substitution in~\eqref{U}.
As mentioned above, $\u$ contains weak BS factors.
Then by Theorem~\ref{RetImplBO}, $\u$ does not satisfy $\Ret$.
\end{example}
Putting together Theorems~\ref{PandPE} and Corollary~\ref{PCrich}, we obtain one more corollary.
\begin{corollary}
Let $\u$ be an infinite ternary word with the language closed under reversal.
Then $$\C \Rightarrow \text{richness of $\u$}.$$
\end{corollary}
In contrast with Corollary~\ref{PCrich}, we see that on a~ternary alphabet already Property $\C$ itself ensures richness.

Neither in this case, the reversed implication holds. Consult Example~\ref{ex:rovnost_bez_C} or the following example with a~periodic word.
\begin{example}[$\text{richness} \not \Rightarrow \C$] \label{ex:rich_bounded_comp}
The periodic infinite word $(abcba)^{\omega}$ is rich
(since return words of palindromic factors are palindromes) and has a~bounded complexity.
\end{example}
\subsubsection{Multiliteral alphabet}
In this section, two new theorems concerning Properties $\P$ and $\PE$ for multiliteral infinite words will be proved.
\begin{thm}\label{th:mr1}
Let $\u$ be an infinite word with the language closed under reversal.
$$\text{Assume $\C$:} \quad \PE \Leftrightarrow \BO.$$
\end{thm}
\begin{proof}
\noindent{($\Leftarrow$):}
Let us prove the statement by contradiction. Assume that Property $\BO$ holds and Property $\PE$ does not.
It is clear that the property $\PE$ can only be violated on a~palindromic BS factor.
By Property $\BO$, all palindromic factors have their bilateral order equal to zero.
By Proposition~\ref{prop_more_than_1}, they have an odd number of palindromic extensions, particularly at least one.

Since the language is closed under reversal, Theorem~\ref{Balazi_fac_pal} implies the inequality \eqref{fac_pal} for all $n \in \N$
\begin{equation*}
\label{PleqC}
\PC{n} + \PC{n+1} \leq 2 + \Delta \C(n).
\end{equation*}

Let $w$ denote the shortest palindromic BS factor that does not have exactly one palindromic extension.
Denote $N = |w|$. Then we have for all $n \leq N$,
$$
\PC{n} + \PC{n+1} = \#\A +1.
$$
Since Property $\BO$ implies Property $\C$, we have $2+\Delta \C(n)=2+(\# \A-1)$,
hence the equality in~\eqref{fac_pal} is attained for all $n\leq N$.

Since $w$ has to have at least $3$ palindromic extensions, one can see that $\PC{N+2} \geq \PC{N} + 2$.
Thus, we obtain $\P(N+1)+\P(N+2)\geq \P(N+1)+\P(N)+2=\# \A+3=\Delta \C(N+1)+4$, which is a~contradiction with \eqref{fac_pal}.
We conclude that Property $\PE$ holds.

\vspace{\baselineskip}
\noindent{($\Rightarrow$):}
Assume Property $\PE$ holds. Then Property $\P$ holds as well. By Corollary~\ref{PCrich} $\u$ is rich.
Consequently, we can apply Theorem~\ref{thm:minus1} and we obtain $\b(w)=0$ for all non-palindromic BS factors and $\b(w)=\#\Pext(w)-1$ for all palindromic BS factors.
By Property $\PE$ every palindromic BS factor has a~unique palindromic extension, thus $\b(w)=0$ for palindromic BS factors, too.
\end{proof}
Let us deduce several Corollaries of Theorem~\ref{th:mr1}. The most straightforward concerns richness and Property $\BO$.
It follows combining Theorems~\ref{th:mr1} and~\ref{rich_opulent}.
\begin{corollary}
Let $\u$ be an infinite word with the language closed under reversal. Then
$$\BO \Rightarrow \text{richness of $\u$}.$$
\end{corollary}

Putting together Theorems~\ref{unif_rec_pal}, \ref{BOImplRet} and~\ref{th:mr1}, we obtain the following corollaries.
\begin{coro}\label{PErichThenRet}
Let $\u$ be a~uniformly recurrent infinite word.
$$\text{Assume $\C$:} \quad \PE \Rightarrow \Ret.$$
\end{coro}
The reversed implication does not hold.
Property $\Ret$ does not even guarantee the weaker property $\P$.
\begin{example}[$\Ret +\C \not \Rightarrow \P$] \label{ex:C_Ret_not_Cur}
Consider again the infinite word from the previous section: the fixed point $\u$ of $\varphi$, where
$\varphi(a)=aab, \ \varphi(b)=ac, \ \varphi(c)=a$. Properties $\C$ and $\Ret$ are satisfied (as explained in~\cite{BaPeSt}),
$\u$ is uniformly recurrent and the language $\L$ is not closed under reversal.
By Theorem~\ref{unif_rec_pal}, $\u$ contains only a~finite number of palindromes.
\end{example}
Notice that the assumptions in Corollary~\ref{PErichThenRet} imply that the language $\L$ is closed under reversal.
It is natural to ask whether the implication $\Ret \Rightarrow \PE$ holds for infinite words with the language closed under reversal.
The answer is however negative.
Property $\Ret$ does not imply even the weaker property $\P$.
\begin{example}($\Ret + \text{reversal closeness} \not \Rightarrow \P$) \label{ex:w2}
Consider again the uniformly recurrent infinite word from~\cite{BaPeSt} given by
$\u=\lim_{n \to \infty}\varphi^n(a)$, where
$$
\varphi(a)=acbca, \ \varphi(b)=acbcadbdaca, \ \varphi(c)=dbcbdacadbd, \ \varphi(d)=dbcbd,
$$
It satisfies $\Ret$, but $\C$ and $\BO$ are violated. It is not difficult to find infinitely many palindromes among weak BS factors.
Thus, the language $\L$ is closed under reversal.  However $\PE$ is not satisfied because $cbc, dbd \in \L$. Nor $\P$ holds since $\P(1)+\P(2)=4 \not =5$.
\end{example}
We see in the previous examples that to demand either only the closeness under reversal or only Property $\C$ in order to reverse the implication in Corollary~\ref{PErichThenRet} is not sufficient.
It is however not solved whether any infinite word with the language closed under reversal and having Properties
$\C$ and $\Ret$ satisfies Property $\PE$ or at least $\P$ as well.

\begin{coro}\label{biinf}
Let $\u$ be a~uniformly recurrent infinite word.
$$\text{Assume $\PE$:} \quad \text{richness of $\u$} \Leftrightarrow \Ret.$$
\end{coro}

\begin{proof}
Recall that by Theorem~\ref{unif_rec_pal}, the language is closed under reversal.\\
$(\Rightarrow)$: Suppose $\u$ is rich.
Then Property $\PE$ guarantees that Property $\P$ holds as well. Property $\P$ and the closeness of $\L$ under reversal together with Theorem~\ref{rich_opulent} implies $\C$ is also satisfied.
The statement follows then by Corollary~\ref{PErichThenRet}.\\
$(\Leftarrow)$: Let us prove the second implication by contradiction.
Assume $\Ret$ is satisfied and $\u$ is not rich. Theorem~\ref{equiv_rich} claims that there exists a~palindrome $w$ which has a~complete return word that is not a~palindrome itself.
As $\PE$ holds, the language has $\# \A + 1$ biinfinite palindromic branches. As $w$ is a~palindrome, we can find it in the middle of one branch.
Since $\u$ is uniformly recurrent, we can find $w$ in a~bounded distance from the center (on both sides) of the remaining $\# \A$ branches.
Thus we have $\# \A$ distinct palindromic complete return words of $w$. As $w$ was supposed to have a non-palindromic return word, we have a contradiction with~$\Ret$.
\end{proof}

In Theorem~\ref{th:mr1} for infinite words having Property $\C$, we have proved that Property $\PE$ coincides with Property $\BO$.
Under the same assumption on the complexity, we are again able to characterize Property $\P$ imposing this time a~weaker condition on bilateral orders of BS factors.
\begin{thm}
\label{th:mr2}
Let $\u$ be an infinite word with the language closed under reversal and satisfying Property~$\C$.
Then Property $\P$ holds if and only if any bispecial factor $w$ of $\u$ satisfies:
\begin{itemize}
\item if $w$ is non-palindromic, then
$$\b(w)= 0,$$
\item if $w$ is a~palindrome, then
$$\b(w)= \# \Pext(w) - 1.$$
\end{itemize}
\end{thm}

\begin{proof}
\noindent{($\Leftarrow$):} Theorem~\ref{thm:minus1} implies that $\u$ is rich. Since the language is closed under reversal, we can use Theorem~\ref{rich_opulent}.
By Property $\C$, we have $\P(n+1)+\P(n)=\Delta \C(n)+2=\#\A+1$, thus Property $\P$ holds.

\vspace{\baselineskip}
\noindent{$(\Rightarrow)$:}
Corollary~\ref{PCrich} states that $\u$ is rich. The statement about bilateral orders follows then by Theorem~\ref{thm:minus1}.
\end{proof}
This theorem may be immediately reformulated using Theorem~\ref{thm:minus1}.
\begin{corollary}
Let $\u$ be an infinite word with the language closed under reversal.
$$\text{Assume $\C$:} \quad \P \Leftrightarrow \text{richness of $\u$}.$$
\end{corollary}

Non-palindromic bispecial factors can really occur in infinite words with the language closed under reversal and satisfying Properties $\C$ and $\PE$, thus $\P$ as well.
This means that there exist rich words with non-palindromic BS factors.

\begin{example} \label{ex:rich_non_pal_BS}
A~ternary word with such properties is ${\mathbf v}=\pi(\u), $ where $\u=\varphi^2(\u)$ and
$$\varphi:A \to CAC, \ B \to CACBD, \ C \to BDBCA , \ D \to BDB, $$
$$\pi: A \to ba, \ B \to b, \ C \to a, \ D \to abc.$$

The substitution $\varphi$ satisfies for any letter $x \in \left \{ A,B,C,D \right \}$,
if we cut off the last two letters of $\varphi^{2n}(x)$, we get a~palindrome.
Together with the uniform recurrence of $\u$, Theorem~\ref{unif_rec_pal} implies that the language $\L$ is closed under reversal.
Every LS factor of $\u$ is a~prefix of $\varphi^{2n}(B)$ or $\varphi^{2n}(C)$ for some $n \in \N$, consequently, $\Delta\C(n)=2$ for all $n \in \N,\ n\geq 1$.

For every non-empty palindrome $w \in \L$, its morphic image $\pi(w)$ without first two letters is a~palindrome.
As $\mathbf v$ contains infinitely many distinct palindromes and is a~morphic image of a~uniformly recurrent word, thus uniformly recurrent, too, the language ${\mathcal L}(\mathbf v)$ is closed under reversal.
The word $\mathbf v$ has two infinite LS branches: every LS factor of $\mathbf v$ is either a~prefix of $\pi(\varphi^{2n}(B))$ or of $\pi(\varphi^{2n}(C))$. Therefore, $\mathbf v$ satisfies Property $\C$. Moreover, $\mathbf v$ contains only ordinary BS factors.
Applying Theorem~\ref{th:mr1}, Property $\PE$ holds as well.
Remark that the factor $ba$ is a~non-palindromic BS factor of~$\mathbf v$.
\end{example}

\subsection{Balance properties}

It is a~direct consequence of the definition that
\begin{equation}\label{ACB}
\AC \ \Rightarrow  \ \B_{\forall} \ \Rightarrow \ \B_{\exists}.
\end{equation}
The first implication follows from the fact that if there are two factors $v,w$ of the same length that contain a~distinct number of letters $a$, say $l$ and $r$, then there exist factors containing any number of letters $a$ between $l$ and $r$ (they may be found in any factor having $v$ as its prefix and $w$ as its suffix).

Let us point out that our favorite generalizations of Sturmian words, namely AR words and $k$-iet words, violate the property $\B_{\forall}$.
The paper~\cite{CaFeZa} provides a~construction of an AR word $\u$ that is not $c$-balanced for any $c$. The same property have also all $3$-iet words given by the transformation $T$ associated with the symmetric permutation and verifying the property i.d.o.c., which can be shown using methods from~\cite{Ad}.

It is natural to ask whether infinite words on multiliteral alphabets with Property $\AC$ exist.
A~recent answer has been provided in~\cite{CuRa}: there are no infinite words satisfying $\AC$
on alphabets containing more than 3 letters.
On the other hand, there exist ternary infinite words with Property $\AC$ as shown by the example taken from~\cite{RiSaZa2}.
\begin{example} \label{ex:psi}
Let $\mathbf v$ be any aperiodic infinite word on $\{A,B\}$ and put $\u=\pi(\mathbf v)$,
where $\pi$ is the morphism defined by $\pi(A)=abc, \ \pi(B)=acb$. Then $\AC(n)=3$
for all $n \in \mathbb N, \ n \geq 1$.
\end{example}
A~more general theorem has been proved ibidem.
\begin{thm}
If an aperiodic uniformly recurrent infinite word $\u$ on a~ternary alphabet is 1-balanced, then $\u$ has Property $\AC$.
\end{thm}

Let us underline in the following examples that none of the implications in~\eqref{ACB} can be reversed.
The first example comes from~\cite{RiSaZa} and the second one is taken from~\cite{Tu}.
\begin{example}[$\B_{\forall} \not \Rightarrow \AC$]\label{ex:Tribonacci}
The ternary Tribonacci word -- the fixed point of the substitution $\varphi: a \to ab, \ b \to ac, \ c \to a$ -- is $2$-balanced, however its abelian complexity reaches five values: $3,4,5,6,7$.
Notice that the Tribonacci word belongs to AR words, which satisfy Properties $\C, \LR, \BO, \Ret, \P, \PE$.
\end{example}

\begin{example}[$\B_{\exists} \not \Rightarrow \B_{\forall}$] \label{ex:beta_coding}
The fixed point $\u$ of the substitution $\varphi: a \to aab, \ b \to c, \ c \to ab$ has the following properties (shown in~\cite{Tu}):
\begin{itemize}
\item for any factors $v,w \in \L$ with $|v|=|w|$, it holds
$$||v|_x-|w|_x|\leq 2 \quad \text{if $x\in \{b,c\}$,}$$
\item there exist $v,w \in \L$ with $|v|=|w|$ such that
$$||v|_a-|w|_a|=3.$$
\end{itemize}
Thus, $\u$ has Property~$\B_{\exists}$.
The word $\u$ is a~coding of distances between neighboring $\beta$-integers, where $\beta$ is the largest root of the polynomial $x^3-2x^2-x+1$.
It is moreover known (see \cite{FrMaPe}) to verify Property $\BO$, but not $\LR$. Theorem~\ref{BOImplRet} implies that $\u$ has Property $\Ret$ as well.
Its language is not closed under reversal, consequently, neither $\PE$ nor $\P$ holds.
\end{example}

Generally, it is difficult to decide whether an infinite word has Property $\B_{\exists}$ or $\B_{\forall}$.
A~slightly simpler problem is to study infinite words that are $c$-balanced for some $c$.
The criterion for existence of such a~constant $c$ for fixed points of a~primitive substitution has been provided in~\cite{Ad2},
observing the spectra of adjacent matrices of substitutions. In general, it is however impossible to determine the minimal value of $c$
from the spectrum. To our knowledge, besides the ternary words considered in Examples~\ref{ex:Tribonacci} and~\ref{ex:beta_coding}, the only non-sturmian
fixed points of primitive substitutions, for which the minimal value of $c$ is known, have been examined in~\cite{BaPeTu} and~\cite{Tu2}.

\section{Overview of relations and examples}

\begin{figure}[h!]

\SelectTips{cm}{12}
\[
\xymatrix@C=2.5cm@R=2cm@M=0.4cm{
\mathcal{BO} \ar@/^5pc/@{->}[rrd]|{\text{CuR}} \ar@/^0.8pc/@{->>}[r]
\ar@/_1.1pc/@{<->}[d]|{\text{CuR}, \mathcal{C}}
\ar@/^1.5pc/@{<-}[rd]|{3} \ar@/^0.3pc/@{<-}[rd]^{4} |-{\object@{/}}
\ar@/_1.1pc/@{->}[rd]|{UR} & \mathcal{C}
\ar@/^1pc/@{->}[rd]|{\text{CuR}, 3} \ar@/_0.1pc/@{->}[d]
|-{\object@{/}} \ar@/_1pc/@{<-}[rd] |-{\object@{/}} & \\
\mathcal{PE} \ar@/^2.5pc/@{->}[ur] |-{\object@{/}}
\ar@/^0.4pc/@{<-}[r]|{3, \text{CuR}} \ar@/_3pc/@{->}[rr] |-{\object@{/}}
\ar@{->}[d] |-{\object@{/}} \ar@/_0.4pc/@{->}[r] |-{\object@{/}}
\ar@/_1.3pc/@{->}[r]|{\text{UR}, \mathcal{C}} \ar@/_1pc/@{->>}[rd] &
\mathcal{R} \ar@/_0.7pc/@{->}[d] |-{\object@{/}} & rich
\ar@/^/@{<->}[l]|{\text{UR}, \mathcal{PE}} \ar@/^/@{->}[dl]
|-{\object@{/}}\\
\text{CuR} & \mathcal{P} \ar@{<->}[r]|{\text{CuR}, \mathcal{C}} &
{\mathcal{BO}_p} \ar@/_0.1pc/@{<->}[u]|{\text{CuR}} }
\]

\vspace{\baselineskip}

\[ \xymatrix@R=0.3\baselineskip{
\ar@{->}[rr] & & & \text{ implication} \\
\ar@{->>}[rr] & & & \text{ irreversible implication} \\
\ar@{<->}[rr] & & & \text{ equivalence} \\
\ar@{->}[rr] |-{\object@{/}} & & & \text{ invalid implication} \\
& \text{number} & & \# \A \\
& \text{UR} & & \text{ uniform recurrence} \\
& \text{CuR} & & \text{ closeness under reversal} \\
& \mathcal{BO}_p & & \b(w) = \left\{ \substack{\displaystyle \# \Pext
(w) - 1 \qquad \text{ if $w$ is bispecial palindromic } \\ \displaystyle
0 \qquad \qquad \qquad \qquad \text{otherwise} \qquad \qquad \qquad
\qquad } \right. \\
}
\]
\caption{Diagram of known relations (assumptions are marked as labels of
arrows)}
\label{diag_relace}
\end{figure}

\crefname{example}{ex.}{ex.}
\newcommand{\Xref}[1]{ex. \ref{#1} on p. \pageref{#1}}

\begin{table}[h!]
\centering
\begin{tabular}[width=\textwidth]{|p{0.35\textwidth}|p{0.4\textwidth}|p{0.15\textwidth}|}
\hline
\textbf{word} & \textbf{properties} & \textbf{reference} \\ \hline \hline
$u_0=ab$, $u_{n+1}=u_n ab \overline{u_n}$ & uniformly recurrent, closed under reversal, finite number of palindromes & \Xref{ex:finite_pals}, \cite{BeBoCaFa} \\ \hline

$ u_0=\varepsilon$, $u_{n+1}=u_n ab c^{n+1} u_n$ & recurrent, $\infty$-many palindromes, not closed under reversal & \Xref{ex:pals_not_closed} \\ \hline

$a \to ab$, $b \to cab$, $c \to ccab$ & $\C$, not $\BO$, not $\Ret$ & \Xref{ex:recoded_chacon}, \Xref{ex:ter_not_R}, \cite{Fer} \\ \hline

$\varphi$: $A \to AB$, $B \to A$; $\pi$: $A \to a$, $B \to bc$ & $\LR$, not closed under reversal, finite number of palindromes, not $\C$, not $\Ret$ & \Xref{ex:lrlr} \\ \hline

$a \to aab$, $b \to ac$, $c \to a$ & $\Ret$, not closed under reversal & \Xref{ex:bapest}, \cite{BaPeSt} \\ \hline

$a \to acbca$, $b \to acbcadbdaca$, $c \to dbcbdacadbd$, $d \to dbcbd$ & $\Ret$, closed under reversal, not $\C$, not $\P$ & \Xref{ex:ter_R_not_C}, \Xref{ex:w2}, \cite{BaPeSt} \\ \hline

$\u=ca\underbrace{cc}_{2\times}b\underbrace{ccc}_{3\times}a\underbrace{cccc}_{4\times}b\underbrace{ccccc}_{5\times}a\dots$ & $\infty$-many palindromes, not closed under reversal, not rich & \Xref{ex:inf_pals_not_rec} \\ \hline

billiard sequence on three letters & closed under reversal, $\PE$, not $\C$, not rich  & \Xref{ex:bo}, \cite{Bo} \\ \hline
$a \to aab$, $b \to aac$, $c \to aa$ &  rich, not $\C$, not $\P$ & \Xref{ex:rovnost_bez_C}, \cite{AmMaPeFr} \\

\hline
$a \to aba$, $b \to cac$, $c \to aca$ & closed under reversal, $\PE$, not $\C$, not $\Ret$ & \Xref{ex:ter_PE_not_C}, \Xref{ex:no_Ret}, \cite{BaPeSta} \\ \hline

$\varphi$: $A \to ABBABBA$, $B \to ABA$; $\pi$: $A \to bc$, $B \to baa$ & closed under reversal, $\C$,$\P$, not $\PE$ & \Xref{ex:ter_P_not_PE}, \cite{BaPeSta} \\ \hline

$(abcba)^{\omega}$ &  rich, not $\C$ & \Xref{ex:rich_bounded_comp} \\ \hline

$a \to aab$, $b \to ac$, $c \to a$ & $\C$, $\Ret$, not closed under reversal & \Xref{ex:C_Ret_not_Cur}, \cite{BaPeSt} \\ \hline

$\varphi$: $A \to CAC$, $B \to CACBD$, $C \to BDBCA$, $D \to BDB$; $\pi$: $A \to ba$, $B \to b$, $C \to a,$ $D \to abc$ & $\PE$, $\C$, closed under reversal, rich, contains non-palindromic BS factors & \Xref{ex:rich_non_pal_BS} \\ \hline

$\u = \pi(\mathbf{v})$, $\pi$: $A \to abc$, $B \to acb$, $\mathbf{v}$ is an aperiodic word over $\{A,B\}$ & $\AC$ & \Xref{ex:psi}, \cite{RiSaZa2} \\ \hline

$a \to ab$, $b \to ac$, $c \to a$ & $\LR, \BO, \Ret, \PE$, $\B_{\forall}$, not $\AC$ & \Xref{ex:Tribonacci}, \cite{RiSaZa} \\ \hline

$a \to aab$, $b \to c$, $c \to ab$ &  $\B_{\exists}$, not $\B_{\forall}$, not closed under reversal, $\BO$, not $\LR$, $\Ret$ & \Xref{ex:beta_coding}, \cite{Tu} \\ \hline

\end{tabular}

\caption{Example overview}
\label{table:overview}
\end{table}

In this section we provide a brief overview of relations and examples presented in the paper.
Most of the relations are depicted in Figure \ref{diag_relace}.
Examples are listed in Table~\ref{table:overview}.
The word is either a fixed point of the given substitution,
the image by the morphism $\pi$ of a fixed point of the substitution $\varphi$,
the limit of the sequence $\left (u_n \right )$ or otherwise specified.

\section{Acknowledgements}
We would like to thank Pierre Arnoux for his valuable remarks and careful reviewing.
We acknowledge financial support by the Czech Science Foundation grant 201/09/0584 and
by the grants MSM6840770039 and LC06002 of the Ministry of Education, Youth, and Sports
of the Czech Republic.

\newpage


\end{document}